\documentclass{article}
\usepackage{amsmath,amsthm,amssymb,amsfonts}
\usepackage{graphicx} 
\usepackage{enumerate}
\usepackage{blkarray}
\usepackage{authblk}
\usepackage{subfigure}
\usepackage{array}
\usepackage{stackengine}
\usepackage{faktor}
\usepackage{fancyhdr}
\usepackage{graphicx}
\usepackage[svgnames]{xcolor}
\usepackage[labelfont=bf,width=0.8\textwidth,justification=centering]{caption}
\usepackage{tikz}
\usepackage{tikz-cd}
\usepackage{subfigure}
\usepackage[colorlinks=true,linkcolor=Indigo,citecolor=Indigo]{hyperref
 }
 \usetikzlibrary{calc,arrows}
\usepackage{mathtools}
\usepackage{color}

\def\p{\mathbf{p}}

\def\Mb{\mathbf{M}}
\def\q{\mathbf{q}}
\def\f{\mathbf{f}}
 \makeatletter
    \newenvironment{sqcases}{%
  \matrix@check\sqcases\env@sqcases
}{%
  \endarray\right.%
}
\def\env@sqcases{%
  \let\@ifnextchar\new@ifnextchar
  \left\lbrack
  \def\arraystretch{1.2}%
  \array{@{}l@{\quad}l@{}}%
}
\makeatother
\newtheorem{theorem}{Theorem}[section]

\newtheorem{corollary}[theorem]{Corollary}
\newtheorem{lemma}[theorem]{Lemma}
\theoremstyle{definition}
\newtheorem{example}[theorem]{Example}

\newtheorem{remark}[theorem]{Remark}

\numberwithin{equation}{section}
  
  \title{Adaptive dynamics of alternating prisoner's dilemma with memory $N$}
\author[1,*]{N. Balabanova}
\author[1]{H. Duong}
\author[2]{C. Hilbe}
\affil[1]{School of Mathematics, University of Birmingham, Birmingham,~UK}
\affil[2]{Interdisciplinary Transformation University (IT:U), Linz, Austria}
\affil[*]{Corresponding author, n.balabanova@bham.ac.uk}
\date{\empty}

\begin{document}
\begin{titlepage}
\small
\maketitle


\begin{abstract}
\noindent
The Prisoner's Dilemma is used as a model in processes involving reciprocity; however, its classical setup can be insufficient in settings where the symmetry of the simultaneous decision making is broken -- for example, in donor and recipient processes. In the alternating Prisoner's Dilemma model the two players take turns choosing their strategy. Assuming a finite memory setup, we establish the mathematical aspects of the adaptive dynamics of the alternating Prisoner's Dilemma, paying particular attention to the case of memory 1. 
\end{abstract}
\footnotetext[1]{MHD and NB were funded by EPSRC (grant EP/Y008561/1). The collaboration between MHD and CH was funded by the Royal International Exchange Grant IES-R3-223047. }
\footnotetext[2]{This manuscript has no associated data}
\end{titlepage}

\normalsize
\tableofcontents

\section{Introduction}
The standard convention for modelling the Iterated Prisoner's dilemma is the simultaneous decision-making of the two players. Indeed, this provides a working model for a large number of interactions. Pehaps, the most well-known ones are stickback behaviour in the presence of a predator (\cite{milinski1987tit}) -- fish examine the predator in pairs, whereas staying back is obviously the safer option-- and  interactions between cleaner fish (Labroides dimidiatus) and their clients; the latter punish bad cleaners that bite rather than clean and reward good cleaners by employing their services again (\cite{bshary2023marine}). 

However, many reciprocal exchanges in nature do not lend themselves well  to modelling in a simultaneous setup. Perhaps the most famous example is the feeding habits of vampire bats (\cite{wilkinson1984reciprocal}) which have been observed donating blood to the starving members in their group. Similar reciprocal food exchanges have been exhibited by Norway rats and brown capuchin monkeys (\cite{schweinfurth2019reciprocity}).  Some birds take turns helping each other defeat their nests (\cite{earl2025cryptic}) and Belgian Shepherd dogs took turns pulling a platform with food for each other (\cite{gfrerer2017working}).  

Nowak  and Sigmund (\cite{nowak1994alternating}) and Park et al (\cite{park2022cooperation}) described the setup for donation games with one "leading" and one "following" player. The former starts the game and makes the first decision in each round; the latter  makes their choice using all available information: their memory about previous turns and the last action of the leading player. This setup is referred to in the literature as the alternating prisoner's dilemma. To distinguish it from the simultaneous case, we will append the epithet ``alternating" to anything pertaining to this model.  

In this paper, we address the challenges of the mathematical setup of adaptive dynamics alternating donation games. One would expect -- and we formulate it as a conjecture at the end of the paper-- that the difference between adaptive  dynamics of alternating and simultaneous prisoner's dilemmas will lessen as the length of memory $N$ increases -- indeed, if every piece of data about the history of the game is equally important, then the significance  of the last choice of the opponent pales in comparison to the rest of the information. However, for memory $N=1$ and $N=2$, some phenomena arise that appear exclusively in the alternating games case. 

 In the first part we set up the general formulation in the spirit of \cite{balabanova2024adaptive}. This is followed by giving a recursive construction for the transition matrix and the payoff vector. Naturally, some of the inherent discrete $\mathbb{Z}_2$-symmetries of the simultaneous case do not persist-- the change of the focal player, in fact, cannot be accomplished by matrix conjugation-- but we show that a subgroup of admissible matrices still is the group of discrete symmetries of the setup. 

We show that the alternating  adaptive dynamics possess the same  $\mathbb{Z}_2$ symmetry as the

In the last section, we examine in detail the  memory 1 game. It comes as a surprise that alternating memory-1 games have two ``one size fits all" invariants (i.e. those that are invariant no matter what the payoff vector is), and due to the form of these functions, their adaptive dynamics can be restricted to two-dimensional tori $\mathbb{T}^2$. We plot the dynamics, estimate the number and location of the equilibria points and discuss their number and stability.

\section{Setup}
\subsection{Adaptive dynamics}
\label{sec: adaptive dynamics}
Adaptive dynamics is one of the most used tools in evolutionary game theory. In the scope of this work, we will give a brief outlook of the setup and refer the reader to \cite{balabanova2024adaptive, stewart:scirep:2016},  and \cite{Mamiya:PRE:2020}. 

Assuming asexual reproduction of the population, we consider some numerical trait $x$ that characterises the majority. Furthermore, the population contains mutants displaying the trait $y$. 

The so-called payoff function $A(y,x)$ (also referred to as invasion fitness) describes the chances of $y$-players survving in the population of $x$-players. 

Mutants successfully invade and subsequently become the majority if and only if
\[
A(y,x)> A(x,x).
\]
This means that mutants' payoff competing with the ``regular" members of the population is higher than that of the ``regular" players competing between themselves. 

When describing invasive species (say, originating in different ecosystems), it makes sense to compare vastly different $x$ and $y$. However, in the context of gradual evolution within a single species the approach of \textit{meso-evolution} (see  \cite{metz2012adaptive})   is more appropriate. 

This approach implies assuming that $x-y<<1$ -- the mutants, despite having evolved, do not differ significantly from the regular members of the population. 

This property of $y$ suggests that with reasonable assumptions on smoothness of $A(x,y)$ in both variables, one may write
\[
A(x,y) = A(x,x) + (y-x)\frac{\partial (A(x,y)}{\partial y}\big\vert_{y=x} + \mathcal{o}(y-x)^2
\]
To maximise the payoff function, the following relation must hold:
\begin{equation}
    \label{eq: adaptive dynamics}
    \dot{y} = \left. \frac{\partial A(y,x)}{\partial y}\right|_{y=x}.
\end{equation}
First introduced by \cite{hofbauer1990adaptive}, these equations are called \textit{adaptive dynamics}; they were first introduced by Nowak and Sigmund \cite{nowak:AAM:1990}. 

As an evolutionary model, the constant evolving is assumed: along the trajectories of solutions of \eqref{eq: adaptive dynamics} individuals that were mutants an infinitesimally small amount of time ago become the ``new norm" and are subsequently replaced by new and better-adapted mutants. 

\subsection{Model}
\label{sec:model}
The setup of this case will rely heavily on the existing and well-studied one for the simultaneous games -- see, for example, \cite{balabanova2024adaptive, laporte2023adaptive}. We will use the same notation that was used there and explain the differences as we encounter them. 

The repeated reciprocal prisoner's dilemma is a repeated donation game between the two players, which we refer to by the $\p$-player and the $\q$-player.  One round of this game consists of both of them making a decision, choosing  between cooperating ($C$) and defecting ($D$). 

As mentioned above, the  setup we consider in this work  differs from the simultaneous one in  having a leading player or a \textit{leader} who gets to choose their action first in every round. Without loss of generality, we can assume this player to be $.\p$, making $\q$ the \textit{follower}.

The payoffs for a single round are set up as follows: $\p$ choosing $C$ means that they receive $a$ and $\q$ receives  $b$. The $p$-player opting for  $D$ means that they receive $c$ and $\q$ gets  $d$. The same happens when it is $\q$'s turn to make a choice. 

For modelling purposes, the  assumption usually  is 
\begin{equation}
\label{eq:conditions}
    \begin{cases}
        c>a,\\
        c-a<b-d
    \end{cases}
\end{equation}, and  some of the numbers can be negative.  

To bring it in accordance with the standard setup of the donation game, we  denote
\begin{equation*}
    \begin{split}
        R&=b+a,\\
        S&=a+d\\
        T&=c+b\\
        P&=c+d.
    \end{split}
\end{equation*}
Then $R,S,T,P$ have the meaning of the total winnings in one round; note that they conform to $R+T = S+P$. This notation allows us to consider each round of the game separately (except for the first one) and treat the payoff vector for alternating games in the same way that we treat the one for simultaneous ones.

The notion of \textit{memory} has to be altered slightly in comparison to the simultaneous case. Previously, saying that the players had \textit{memory} $N$ it was equivalent  to saying that  $\p$ and $\q$ remembered $N$ previous rounds -- this meant that both kept in mind the previous $2N$ actions. Now, since decision-making is staggered, we need to define the notion of memory slightly more carefully. We will still say that having memory $N$ means remembering $2N$ previous choices, but this will now have slightly different meanings for $\p$ and $\q$. 

That is, for $\p$ this is still the same as remembering the last $N$ rounds, since they make the choice first. On the other hand, for $\q$ this means remembering their decision from $N$ rounds ago, previous $N-1$ rounds, and the decision that the $\p$-player just made.

To define $\p$ and $\q$, we turn again  to \cite{balabanova2024adaptive} and CITE. Assuming that the players have memory $N$,  $\p,\q\in\mathbb{R}^{2N}$. These are the vectors of conditional probability of cooperation, conditional on the outcomes that the players remember. We refer to both of these vectors as \textit{strategies}.

Similar to \cite{balabanova2024adaptive}, we will denote the entries of these vectors  by $p_{i_1 i_2\ldots i_{2N-1}i_{2N}}$ and $q_{j_1 j_2\ldots j_{2N-1}j_{2N}}$, where $i_k, j_k\in\{C,D\}$. For the $\p$player, each pair of consecutive indices $i_{2k-1}i_{2k}$ is the encoding of the two decisions made  $k$ rounds ago, with the assumption that the $\p$-player's choice comes first.

For the $\q$-player, $i_1$ is their decision $N$ rounds ago, $i_{2k} i_{2k+1}$ are the choices from $k$ rounds ago, and $i_{2N}$ is the last decision of $\p$. 

\begin{example}
    Suppose that $N=3$. Then $p_{CDDCDD}$ denotes the conditional probability of the $\p$-player cooperating in the subsequent round, if three rounds ago $\p$ cooperated and $\q$ defected, two rounds ago $\p$ defected and $\q$ cooperated, and in the previous round both $\p$ and $\q$ defected. 

    With the same assumption on $N$, consider the element $q_{CDCCDC}$. This variable denotes the conditional probability of $\q$'s cooperation in the game with the following history: three rounds ago $\q$ cooperated, two rounds ago $\p$ defected and $\q$ cooperated, one round ago $\p$ cooperated and $\q$ defected, and in the current round $\p$ has already made a choice and cooperated. 
\end{example}

For vectors $\p$ and $\q$ we  arrange the entries in the lexicographic order, with $C$ coming before $D$:
 \[
\p = \begin{pmatrix}
    p_{CC\ldots CC}\\
    p_{CC\ldots CD}\\
    p_{CC\ldots DC}\\
    p_{CC\ldots DD}\\
    \vdots\\
    p_{DD\ldots DD}\\
\end{pmatrix}, \ \ \q = \begin{pmatrix}
    q_{CC\ldots CC}\\
    q_{CC\ldots CD}\\
    q_{CC\ldots DC}\\
    q_{CC\ldots DD}\\
    \vdots\\
    q_{DD\ldots DD}\\
\end{pmatrix}
\]
\begin{remark}
    In \cite{balabanova2024adaptive}, we used quite extensively our ability to ``pair up" elements of $\p$ and $\q$: by exchanging indices $i_{2k-1}$ and $i_{2k}$ for some $2N$-turple of an elenemt of $\p$, we can  get the number of the element of $\q$ that told the same story of the game. 

    In this setup, this property is obviously lost -- the elements of $\p$ contain some information that the $\q$-player has already forgotten, and $\q$ takes into account the last choice of $\p$. 
\end{remark}

For brevity, we will denote elements of vectors $\p$ and $\q$ as $p_{i}$ and $q_{j}$.

The process of repeated alternating games is a Markov chain process for any $N$ (see \cite{nowak1994alternating} for memory 1, and we will prove this statement below) with a transition matrix, which we denote $\Mb_N(\p,\q)$. The states of this chain are the outcomes of the previous rounds, and it can be clearly seen that the matrix $\Mb_N$ is a left stochastic matrix. 

Consequently, it has a left eigenvector $\nu_N(\p,\q)$ with an eigenvalue 1, i.e. a vector such that $\nu_N^T(\p,\q)\Mb_N(\p,\q) = \nu^T(\p,\q)$. This vector is referred to as the invariant distribution (\cite{norris1998markov}). 

Above, we have briefly discussed the payoffs $R,S,T,P$ from one round of the game. Identically to the simultaneous case, we can construct a vector $\f_N$ that encodes in lexicographic order the winnings in $N$ previous rounds. 
\begin{lemma}
\label{st: payoff vector correctly defined}
    The vector $\f_N$ is well-defined, i.e. the winnings of the two players are equal for identical game histories. 
\end{lemma}
\begin{proof}
    We start with the toy case $N=1$. Then the two strategy vectors and the payoff vectors of the winnings of $\p$ and $\q$ respectively  are
    \[
    \p:=\begin{pmatrix}
        p_{CC}\\
        p_{CD}\\
        p_{DC}\\
        p_{DD}
    \end{pmatrix}, \ \ \q:=\begin{pmatrix}
        q_{CC}\\
        q_{CD}\\
        q_{DC}\\
        q_{DD}
    \end{pmatrix}, \ \ \f_{\p}:=\begin{pmatrix}
        f_{CC}^{\p}\\
        f_{CD}^{\p}\\
        f_{DC}^{\p}\\
        f_{DD}^{\p}
    \end{pmatrix}, \ \ \f_{\q}:=\begin{pmatrix}
        f_{CC}^{\q}\\
        f_{CD}^{\q}\\
        f_{DC}^{\q}\\
        f_{DD}^{\q}
    \end{pmatrix}
    \]

    Consider $f_{\p}$ first. The entry $f_{CC}^{\p}$ denotes the winnings of $\p$ after one complete round, where $\p$ and $\q$ both chose to cooperate. On the other hand, the element $f^{\q}_{CC}$ encodes the winnings of $\q$ after $\q$ decided to cooperate in the previous round and $\p$ cooperated in the current round. Using the notation from REF, we will compute both vectors:
    \[
    f_{\p} =\begin{pmatrix}
        a + b\\
        a + d\\
        c + b\\
        d + c
    \end{pmatrix},    f_{\q} =\begin{pmatrix}
      a + b\\
       a + d\\
       c+b\\
       c+d
    \end{pmatrix}
    \] 
    and $f_{\p} = \f_{\q}$. 

    For general $N$ (as, actually, for $N=1$) the proof is very straightforward. Consider two elements of $\p$ and $\q$ with the same indices: $p_{i_1i_2\ldots i_{2N}}$ and $q_{i_1i_2\ldots i_{2N}}$.  the index of $\p$ encodes first action of $\p$, then action of $\q$, then action of $\p$ and so on.  But in the index of $\q$ their action comes first, then goes the action of $\p$, and, again, so on. 

    Due to how the payoffs per choice are defined and their inherent symmetries, the payoffs after $2N$ rounds are the same. 
\end{proof}

To define adaptive dynamics, we apply the method identical to CITE. Define the \textit{payoff function}, i.e. the expected winnings of players, as 
\begin{equation}
    \label{eq: defining payoff function}
    A(\p,\q) := \left<\nu_N(\p,\q), \f_N\right>. 
\end{equation}
The works \cite{hilbe2017memory, laporte2023adaptive, balabanova2024adaptive} provide a recipe for writing the payoff function explicitly; we provide a similar lemma here. Since the relevant properties of the transition matrices are identical in the simultaneous and alternating cases, we do not provide the proof in the scope of this work. 

\begin{lemma}[\cite{hilbe2017memory}]
\label{st: payoff function}
Consider a transition matrix $\mathbf{M}_N$ for an alternating game with memory $N$ and subtract the identity matrix $I$ of the appropriate dimension from it. Replace the last column by the vector  $\mathbf{f}_N$ and denote this matrix by $\Mb_N^1$. To obtain the matrix $\Mb_N^2$, substitute $\mathbf{1}_N$ for $\mathbf{f}_N$ in $\Mb_N^1$.  Then the payoff function $A(\p,\q)$ is given by 
\begin{equation}
    \label{eq:payoff function}
A(\mathbf{p},\mathbf{q}) = \frac{\begin{vmatrix}\Mb_N^1\end{vmatrix}}{\begin{vmatrix}\Mb_N^2\end{vmatrix}}, 
\end{equation}
where $|M|$ is the determinant of the matrix $M$.
\end{lemma}
Lastly, as per the standard definition, the \textit{adaptive dynamics} for the strategy $\p$ will be given by 
\[
\dot{\p}:=\frac{\partial A(\p,\q)}{\partial \p}\Big\vert_{\p=\q}.
\]

\subsection{Transition matrix}
\label{sec:transition matrix}
As mentioned above, the repeated donation game with memory $N$ is a Markov process.
The transition matrix for memory 1 is  given in the work \cite{nowak1994alternating} and has the form 
\begin{equation}
    \label{eq: transition memory 1}
    \Mb_1{\p,\q} = \begin{pmatrix}
        p_{CC}q_{CC} & p_{CC}(1-q_{CC})& (1-p_{CC})q_{CD}& (1-p_{CC})(1-q_{CD})\\
         p_{CD}q_{DC} & p_{CD}(1-q_{DC})& (1-p_{CD})q_{DD}& (1-p_{CD})(1-q_{DD})\\
           p_{DC}q_{CC} & p_{DC}(1-q_{CC})& (1-p_{DC})q_{CD}& (1-p_{DC})(1-q_{CD})\\
         p_{DD}q_{DC} & p_{DC}(1-q_{DC})& (1-p_{DD})q_{DD}& (1-p_{DD})(1-q_{DD})\\
         \end{pmatrix}
\end{equation}
Indeed, one can check that 
\begin{small}
\begin{equation*}
\begin{split}
  &  \begin{pmatrix}
\nu_{CC}(n+1)&\nu_{CD}(n+1)& \nu_{DC}(n+1)& \nu_{DD}(n+1) \end{pmatrix}=\\& \begin{pmatrix}
\nu_{CC}(n)&\nu_{CD}(n)& \nu_{DC}(n)& \nu_{DD}(n)\end{pmatrix}\begin{pmatrix}
        p_{CC}q_{CC} & p_{CC}(1-q_{CC})& (1-p_{CC})q_{CD}& (1-p_{CC})(1-q_{CD})\\
         p_{CD}q_{DC} & p_{CD}(1-q_{DC})& (1-p_{CD})q_{DD}& (1-p_{CD})(1-q_{DD})\\
           p_{DC}q_{CC} & p_{DC}(1-q_{CC})& (1-p_{DC})q_{CD}& (1-p_{DC})(1-q_{CD})\\
         p_{DD}q_{DC} & p_{DC}(1-q_{DC})& (1-p_{DD})q_{DD}& (1-p_{DD})(1-q_{DD})\\
         \end{pmatrix} 
         \end{split}
\end{equation*}
\end{small}

In particular, for example, 
\begin{equation*}
\begin{split}
\nu_{CC}(n+1) = \nu_{CC}(n) p_{CC}q_{CC} + \nu_{CD}(n)p_{DC}q_{DC} + \nu_{DC}(n)p_{DC}q_{CD} + \nu_{DD}(n)p_{DD}q_{DC}.
\end{split}
\end{equation*}
Analogous relations hold for the rest of the components.

Now we proceed to construct the transition matrix for memory $N$.

It is clear that the transition matrix will have the same four-diagonal structure as the one for the regular game. This follows from the expression of  $\nu_{i_1 i_2i_3 i_4\ldots i_{2N-1} i_{2N}}(n+1)$ as a linear combination of $\nu_{CC\ldots i_{2N-3} i_{2N-2}}(n)$,  $\nu_{CD\ldots i_{2N-3} i_{2N-2}}(n)$, $\nu_{DC\ldots i_{2N-3} i_{2N-2}}(n)$ and $\nu_{DD\ldots i_{2N-3} i_{2N-2}}(n)$. 

The particulars of linear combination depend on the last two elements of the indices $i_{2N-1}$ and $i_{2N}$. 

For example, if $i_{2N-1} = C i_{2N}= D$, we have, as in CITE
\begin{small}
    \begin{equation}
        \begin{split}
         \nu_{i_1 i_2i_3 i_4\ldots CD}(n+1)  &= \nu_{CC\ldots i_{2N-3} i_{2N-2}}(n) p_{CC\ldots i_{2N-3} i_{2N-2}}(1- q_{C\ldots i_{2N-2} C}) \\&+ \nu_{CD\ldots i_{2N-3} i_{2N-2}}(n) p_{CD\ldots i_{2N-3} i_{2N-2}}(1- q_{D\ldots i_{2N-2} C})\\&+\nu_{DC\ldots i_{2N-3} i_{2N-2}}(n) p_{DC\ldots i_{2N-3} i_{2N-2}}(1- q_{C\ldots i_{2N-2} C})\\&+\nu_{DD\ldots i_{2N-3} i_{2N-2}}(n) p_{DD\ldots i_{2N-3} i_{2N-2}}(1- q_{D\ldots i_{2N-2} C})
        \end{split}
    \end{equation}
\end{small}
Note that the last index of all elements of $\q$ in the formula above must be $C$, since the $\p$-player is assumed to have cooperated. 

We can conclude something else from the formula above: that not only is the four-diagonal structure of the transition matrix preserved, the placement of the elements of $\p$ has to be the same as in the simultaneous case. However, it behaves differently in $\q$.

In summary, when it is the $\q$-player's turn to choose, they ``forget" the choice $i_1$.  This entails that the list of remembered choices starts with the action of the $q$-strategist,and one only needs to append the last decision of the $p$-strategist to it. 

Therefore, if we denote $\underline{i} = i_2\ldots i_{2N}$, the quadruples in the matrix $M(\p,\q)$ will have the form 
\begin{small}
\[
p_{i_1\underline{i}}q _{\underline{i}C} \ \ \ p_{i_1\underline{i}}(1-q _{\underline{i}C}) \ \ \  (1-p_{i_1\underline{i}})q _{\underline{i}D} \ \ \ (1-p_{i_1\underline{i}})(1-q _{\underline{i}D}) 
\]
\end{small}

We assmeble these observations in the following lemma

\begin{lemma}
\label{lem: construction matrix}
For $N\!\ge\!2$, let the $2^{2(N-1)}\times2^{2(N-1)}$ transition matrix for the alternating memory-$(N\!-\!1)$ game have the form
\begin{equation}
    \label{eq: memory N-1 matrix}
  \mathbf{M}_{N-1}:=  \begin{pmatrix}
        M_1\\
        M_2\\
        M_3\\
        M_4
    \end{pmatrix}
\end{equation}
where $M_1, M_2, M_3, M_4$ are $2^{2(N-2)}\times2^{2(N-1)}$ matrices that are obtained by  cutting the matrix $\mathbf{M}_{N-1}$ horizontally into four submatrices of equal dimensions. 

Then the  transition  matrix for the  memory $N$ game has the following form:
\begin{equation}
    \label{eq: memory N matrix}
 \mathbf{M}_N  =   \begin{pmatrix}
        M'_1& 0& 0&0\\
        0&M'_2&0&0\\
        0&0&M'_3&0\\
        0&0&0&M'_4\\
        M''_1& 0& 0&0\\
        0&M''_2&0&0\\
        0&0&M''_3&0\\
        0&0&0&M''_4\\
        M'''_1& 0& 0&0\\
        0&M'''_2&0&0\\
        0&0&M'''_3&0\\
        0&0&0&M'''_4\\
        M''''_1& 0& 0&0\\
        0&M''''_2&0&0\\
        0&0&M''''_3&0\\
        0&0&0&M''''_4
    \end{pmatrix}.
\end{equation}
Here,  "0" denotes a $2^{2(N-2)}\times2^{2(N-1)}$  matrix consisting of zeros. The dashed matrices are constructed in the following way: to the beginning of the list of indices of each $p_{i_1\ldots i_{2(N-1)}}$ from any of the matrices $M_i$ we append   $CC$, $CD$, $DC$ or $DD$, to make it respectively into $M'_i$, $M''_i$, $M'''_i$ or $M'''_i$. To  the indeices of $q_{i_1\ldots i_{2(N-1)}}$  we append the second and the third indices of the new $\p$. 
\end{lemma}

For example, if  one of the components of $\p$ used to have the index $p_{CC\   CD \  DD}$ and was in the first element of a quadruple, then its corresponding component of $\q$ will be $q_{CC\ DD\ DC}$. If we add, say, $CD$ to the index of $p_{CC\   CD \  DD}$, then we need to add $D$ to $q_{CC\ DD\ DC}$ at the beginning and $C$ or $D$ at the and, depending on whether it lies in the first/second or third/forth element of the quadruple.

\begin{remark}
This matrix is principally different from the matrix for the simultaneous game. In the former case, each $\p$-element corresponds to one and only one $\q$-element, and both of them occur only in one quadruple. Here, in each quadruple, two $\q$-elements correspond to one $\p$-element, and therefore, $\q$-elements ``run out" twice as fast. 

Therefore, the matrix structure in $\q$ is repeated, i.e. if we divide the matrix horizontally in half, we have the same pattern in $\q$. Additionally, if one goes along rows from left to right the index of $\q$ increases with every two quadruples.

The elements of $\p$ are, hover, placed in the same positions as they were in the transition matrix for the simultaneous game. 

Thanks to Lemma \ref{st: payoff vector correctly defined}, we can think of the payoff vector $\mathbf{f}_N$  in terms of the variables $R,S,T$ and $P$ -- and construct the payoff vector using only the winnings of the $\p$-player. This allows us to formulate the following 
\begin{lemma}
\label{st: recursive payoff vector}
Given the payoffs $R,S,T,P$ of the prisoner's dilemma, the payoff vector $\mathbf{f}_N$ for the memory-$N$ game can be constructed recursively. Let 
    \begin{equation}
    \label{eq: recursive noq quite payoff}
    \begin{split}
       \widetilde{ \mathbf{f}}_1  = \begin{pmatrix} R\\S\\T\\P \end{pmatrix}, \ \ \widetilde{\mathbf{f}}_{N} = \begin{pmatrix} \widetilde{\mathbf{f}}_{N-1}+R\mathbf{1}_{N-1}\\\widetilde{\mathbf{f}}_{N-1}+S\mathbf{1}_{N-1}\\\widetilde{\mathbf{f}}_{N-1}+T\mathbf{1}_{N-1}\\\widetilde{\mathbf{f}}_{N-1}+P\mathbf{1}_{N-1}\end{pmatrix}, 
        \end{split}
    \end{equation}
    
    where $\mathbf{1}_N$ is a $2^{2N}$ dimensional vector whose entries are all equal to 1. Then 
    \begin{equation}
        \mathbf{f}_N =\frac{1}{N}\widetilde{\mathbf{f}_N}.
    \end{equation}
\end{lemma}
\begin{proof}
    As we have stated above, we can construct the payoff vector just using the payoffs of $\p$.  

    Then the entire construction will be exactly the same as that in CITE. 
\end{proof}

\end{remark}

\subsection{Symmetries of the setup}
\label{sec:symmetries setup}
We follow the introduction of the mathematical setup by the same considerations as in CITE. donation games have some inherent symmetries: for simultaneous games, we can change the focal player, as well as change the notion of cooperation and defection for either or both players. 

In this section, we set out to uncover the mathematical formulation of the symmetries of an alternating game. Again, we leverage on the results of CITE.

The realisation of the symmetries will be as follows. 
Consider an orthogonal matrix  $\mathbf{O}$  and the transformations of the form

\begin{equation} \label{Eq:OTransformation}
\nu\mapsto \mathbf{O}\nu, \ \mathbf{M}_N\mapsto \mathbf{O}\mathbf{M}_N\mathbf{O}^T, \mathbf{f}_N\mapsto \mathbf{O}\mathbf{f}_N.
\end{equation}
to all elements of the game. 
\begin{lemma}[\cite{balabanova2024adaptive}]
\label{st: lemma classification}
  The transformation \eqref{Eq:OTransformation} with an orthogonal matrix $\mathbf{O}$ leaves the payoff function invariant.
\end{lemma}
Similarly, we demand that the matrix $\mathbf{O}\Mb\mathbf{O}^T$ is still a transition matrix of a game. If we call such matrices $\mathbf{O}$ \textit{admissible}, the following immediately follows: 
\begin{lemma}
      Any admissible orthogonal matrix $\mathbf{O}$ is a permutation matrix. 
\end{lemma}

Clearly, we can not realise the change of the leading player through the means of matrix multiplication, as we did for simultaneous games. The matrices will have to be permutation matrices for the same reasons. But for every $\p_i$ in each line we have two different $\q_j$s, and for every $\q_j$ the two $\p_i$s are in different rows. Therefore, no exchange can be done via the means of permutation.

\subsubsection{$\mathbb{Z}_2$ symmetries}
\begin{lemma}
\label{st: matrices memory 1}
The following matrices are admissible (in the sense of CITE) for the memory 1 alternating game.  
\begin{equation}
        \label{eq: good J for memory 1}
        \begin{split}
 &  J^1_1 := \left(
\begin{array}{cccc}
 1 & 0 & 0 & 0 \\
 0 & 1 & 0 & 0 \\
 0 & 0 & 1 & 0 \\
 0 & 0 & 0 & 1 \\
\end{array}
\right), \ J^2_1:=\left(
\begin{array}{cccc}
 0 & 1 & 0 & 0 \\
 1 & 0 & 0 & 0 \\
 0 & 0 & 0 & 1 \\
 0 & 0 & 1 & 0 \\
\end{array}
\right), \ J^3_1:=\left(
\begin{array}{cccc}
 0 & 0 & 1 & 0 \\
 0 & 0 & 0 & 1 \\
 1 & 0 & 0 & 0 \\
 0 & 1 & 0 & 0 \\
\end{array}
\right), \\& \ J^4_1:=
\left(
\begin{array}{cccc}
 0 & 0 & 0 & 1 \\
 0 & 0 & 1 & 0 \\
 0 & 1 & 0 & 0 \\
 1 & 0 & 0 & 0 \\
\end{array}
\right).
        \end{split}
    \end{equation}
    \end{lemma}
    \begin{proof}
        This statement can be checked by an explicit computation. The transition matrix for memory 1 has the form
        \[
        \Mb_1(\p,\q) = \begin{pmatrix}
            p_1 q_1& p_1(1-q_1)& (1-p_1)q_2& (1-p_1)(1-q_2)\\
             p_2 q_3& p_2(1-q_3)& (1-p_2)q_2& (1-p_1)(1-q_2)\\
              p_3 q_1& p_3(1-q_1)& (1-p_3)q_2& (1-p_3)(1-q_2)\\
                p_4 q_3& p_4(1-q_3)& (1-p_4)q_2& (1-p_4)(1-q_2)\\
        \end{pmatrix}
        \]
       We omit the transposition in what follows, since all the matrices $J_1^i$ are invariant with regards to it. 
        
    Then the conjugation yields the following  results:
        \[
        J_1^2\Mb_1J_1^2 = \left(
\begin{array}{cccc}
 p_2 (1-\text{q3}) & p_2 \text{q3} & (1-p_2) (1-\text{q4}) & (1-p_2) \text{q4} \\
 \text{p1} (1-\text{q1}) & \text{p1} \text{q1} & (1-\text{p1}) (1-\text{q2}) & (1-\text{p1}) \text{q2} \\
 p_4 (1-\text{q3}) & p_4 \text{q3} & (1-p_4) (1-\text{q4}) & (1-p_4) \text{q4} \\
 \text{p3} (1-\text{q1}) & \text{p3} \text{q1} & (1-\text{p3}) (1-\text{q2}) & (1-\text{p3}) \text{q2} \\
\end{array}
\right),
        \]
        meaning that $J_1^2$  exchanges $p_{CD}$ with  $p_{CC}$, $p_{DC}$ with $p_{DD}$, $q_{CC}$ with $1-q_{CD}$ and $q_{CD}$ with $1-q_{DD}$. As a symmetry of the game, it exchanges the notions of cooperation and defection for the  $\q$-player. 
\[
        J_2^2\Mb_1J_2^2 = 
\left(
\begin{array}{cccc}
 (1-\text{p3}) \text{q2} & (1-\text{p3}) (1-\text{q2}) & \text{p3} \text{q1} & \text{p3} (1-\text{q1}) \\
 (1-p_4) \text{q4} & (1-p_4) (1-\text{q4}) & p_4 \text{q3} & p_4 (1-\text{q3}) \\
 (1-\text{p1}) \text{q2} & (1-\text{p1}) (1-\text{q2}) & \text{p1} \text{q1} & \text{p1} (1-\text{q1}) \\
 (1-p_2) \text{q4} & (1-p_2) (1-\text{q4}) & p_2 \text{q3} & p_2 (1-\text{q3}) \\
\end{array}
\right).        \]
        Here, $p_1$ is substituted by $1-p_3, $ $p_2$ by $1-p_4$, $q_1$ by $q_2$ and $q_3$ by $q_4$. This signals the  exchange of notions of cooperation and defection for the $\p$-player. 
        
 Lastly, we observe that $J_1^2J_1^3 = J_1^3J_1^2 = J_1^4$-- this entails that conjugation by the matrix $J_1^4$ exchanges cooperation and defection for both players.
    \end{proof}
   
    Consider the recursively constructed matrices as before: let $J^i_1$ be the matrices from Lemma \ref{st: matrices memory 1}, and at $N$th step  in the iteration we put a $2^{2(N-1)}$ matrix containing of $0$s instead of each $0$ and the matrix $J^i_{N-1}$ instead of a 1 in $J^i_{N-1}$.

Then following holds:
\begin{theorem}
    The only admissible matrices for memory $N$-alternating games are $J^1_N$, $J^2_N,  J^3_N$ and $J^4_n$. 
\end{theorem}
\begin{proof}
    Clearly, the matrices that are admissible need to be contained within the set of admissible matrices for simultaneous games, since we were primarily concerned with the form of the matrix when we constructed them. 

    However, any of the matrices that change the focal player are clearly not admissible; additionally, we can only exchange 1st and 2nd blocks, as well as 3rd and 4rth; this excludes any of the admissible matrices but the ones listed in the statement of the theorem. 
\end{proof}
\subsection{Symmetries of the dynamics}
\label{sec: symmetries dynamics}
In this section we briefly discuss some more general properties of the adaptive dynamics: namely, the same $\mathbb{Z}_2$-symmetry that the simultaneous dynamics possess. 

In \cite{balabanova2024adaptive}, the connection was spelled out between  the vectors $\mathbf{f}_N$, $\mathbf{1}_N$ and $J_N^4\mathbf{f}_N$ for games with equal gains from switching. 

By the very construction in Section CITE, every alternating game yields a payoff vector with equal gains from switching. Therefore, the following identity holds in our case as well:
\begin{lemma}
    \label{st: multiplication payoff by J}
    \begin{equation}
        -\mathbf{f}_N + C_N\mathbf{1}_N = \mathbf{J}_N^8\mathbf{f}_N,
    \end{equation}
    where $C_N$ is a constant depending on $N$.
\end{lemma}

For brevity we do not provide the detailed proof here, since it is identical to the one contained in \cite{balabanova2024adaptive}-- which we refer the reader to. 

Introduce a transformation $\varphi:[0,1]^{2^{2n}}\rightarrow[0,1]^{2^{2n}}$, defined by
$$\varphi\big( \p \big) = \varphi\big( (p_{C\ldots C},\ldots,p_{D\ldots D})\big) = (1-p_{D\ldots D},\ldots,1\!-\!p_{C\ldots C}).$$. Clearly, this is the exchange of notions of winning or losing -- performing the same function as the matrix $J_N^3$ from Section CITE does (or $J_N^2$, if we consider $\varphi(\q)$).

\begin{theorem}
If $(p_{CC\ldots C}(t),\ldots, p_{DD\ldots D}(t))$ is a trajectory of the adaptive dynamics for a given alternating donation game, then so is $(1- p_{DD\ldots D}(-t),\ldots, 1-p_{CC\ldots C}(t))$.  
\end{theorem}
\begin{proof}
First, we demonstrate that the following four actions are equivalent:
\begin{enumerate}
  \item  $\Mb_N(\p,\q)\mapsto \Mb_N(\varphi(\p),\varphi(\q))$
    \item $\Mb_N(\p,\q)\mapsto J_N^8\Mb_N(\p,\q)J_N^8$
    \item $\Mb_N(\p,\q)\mapsto (\Mb(\p,\q)^T)^{\tau}$
\end{enumerate}
Here, the operation $T$ is the standard matrix operation of transposition, and $\tau$ is the transposition with respect to the skew diagonal. 

The equivalence of the second and third statements is fairly obvious: multiplying by $J_N^8$ from the left reverses the order of the rows, and multiplying by $J_N^8$ from the right reverses the order of the columns. This can easily be seen to be equivalent to the combination of the two transpositions.

Therefore, to complete our proof, we need the equivalence of, say, the first and the third statement. Compared to simultaneous games, this is a simpler case. Every row of $\Mb_N(\p,\q)$  is indexed by its own element of $\p$. Imagine that we are moving along a matrix from left to right, top to bottom, jumping to the subsequent row as we finish tracin the previous one. 

Then this way, we ``go over" the $\p$-vector once in lexicographic order and over the $\q$-vector also twice, in lexicographic order. Applying the two transpositions to the matrix is equivalent to tracing the rows and columns in the opposite direction, and evidently yields the matrix of the same form, only dependent on $\p$ and $\q$ written in the reverse order. 

This proof is completed by observing two things:
\begin{enumerate}
    \item The four-diagonal structure of the matrix is preserved. An element $m_{i\ j}$ of $\Mb_N(\p,\q)$ is mapped to the element $m_{n-i \ n-j}$, and the quadruples that were together stay together. 
    \item the order of the elements in the quadruples is reversed -- this guarantees that the double transposed matrix $\Mb$ is the function of $\varphi(\p)$, $\varphi(\q)$ as opposed to just multiplicands of $\p$ and $\q$. 
\end{enumerate}
Since we have shown that 1 and 3 are equivalent, all three are equivalent. 

The rest of the proof is identical to that in CITE. Introducing the new payoff vector
\[
g(\mathbf{f}_N):= -J_N^8\mathbf{f}_n =K_N\mathbf{1}_N + \mathbf{f}_N 
\]
for some constant $K_N:=K_N(N)$ and the new time
\[
\kappa:=-t, 
\]
we can write for $A(\varphi(\p),\varphi(\q),g(\mathbf{f}_N)$ and $\tau$:
\begin{equation}
    \label{eq: reversed everything}
    \begin{split}
    \frac{\partial A(\varphi(\p),\varphi(\q), \mathbf{f}_N) }{\partial\varphi(\p)}\Big|_{\varphi(\p) =\varphi(\q)} &= \frac{\partial \left<\nu_{\varphi(\p), \varphi(\q)},\mathbf{f}_N\right>}{\partial\varphi(\p)}\Big|_{\varphi(\p) =\varphi(\q)} \\&=\frac{\partial\left<\mathbf{J}_N^8\nu_{\p,\q},\mathbf{f}_N\right>}{\partial\mathbf{p}}\frac{\partial\mathbf{p}}{\partial{\varphi(\p)}}\Big|_{\p=\q} \\&= -(\mathbf{J}^8)\frac{\partial\left<\nu_{\p,\q},\mathbf{J}^8\mathbf{f}_N\right>}{\partial\mathbf{p}}\Big|_{\p=\q} \\&=-(\mathbf{J}^8)\frac{\partial\left<\nu_{\p,\q},-\mathbf{f}_N\right>}{\partial\mathbf{p}}\Big|_{\p=\q} =  \mathbf{J}^8\dot{\mathbf{p}}.
\end{split}
\end{equation}
Thus, if we denote the left hand side of the equations by $F(\p)$, we get 
\begin{equation}
F(\varphi(\p)) = F(\mathbf{1} - \mathbf{J}_n^8\p) = \mathbf{J}_N^8F(\p) =  - D(\mathbf{1} -\mathbf{J}_N^8\p)F(\p),
\end{equation}
and that is exactly the symmetry condition, if we combine it with the time reversal.
\end{proof}

\section{Memory 1}
\label{sec: memory 1}
From this point onwards, we set the memory length equal to one and discuss the properties of adaptive dynamics. 

As before, we assume that the $p$-player is the ``leader" and that the $q$-player, the  ``follower". For brevity, we will use the notation $(p_{CC}, p_{CD},p_{DC},p_{DD}) = (p_1,p_2, p_3, p_4)$ (analogously, for $q$). 

In Accordance with CITE, the  transition matrix  has the form
\begin{equation}
    \label{eq: matri memory 1}
   M_1(\p,\q) =  \left(
\begin{array}{cccc}
 p_1 q_1 & p_1 (1-q_1) & (1-p_1) q_2 & (1-p_1) (1-q_2) \\
 p_2 q_3 & p_2 (1-q_3) & (1-p_2) q_4 & (1-p_2) (1-q_4) \\
 p_3 q_1 & p_3 (1-q_1) & (1-p_3) q_2 & (1-p_3) (1-q_2) \\
 p_4 q_3 & p_4 (1-q_3) & (1-p_4) q_4 & (1-p_4) (1-q_4) \\
\end{array}
\right)
\end{equation}
Following the recipe from Section \ref{sec:model}, we construct the payoff function:
\begin{small}
   \begin{equation}
  A(\mathbf{p},\mathbf{q}) =      \frac{\left| 
\begin{array}{cccc}
 p_1 q_1-1 & p_1-1 & (1-p_1) q_2 & f_1 \\
 p_2 q_3 & p_2-1 & (1-p_2) q_4 & f_2 \\
 p_3 q_1 & p_3 & (1-p_3) q_2-1 & f_3 \\
 p_4 q_3 & p_4 & (1-p_4) q_4 & f_4 \\
\end{array}
\right| }{\left| 
\begin{array}{cccc}
 p_1 q_1-1 & p_1-1 & (1-p_1) q_2 & 1 \\
 p_2 q_3 & p_2-1 & (1-p_2) q_4 & 1 \\
 p_3 q_1 & p_3 & (1-p_3) q_2-1 & 1 \\
 p_4 q_3 & p_4 & (1-p_4) q_4 & 1 \\
\end{array}
\right| }
   \end{equation}
\end{small}

To discuss the dynamics in more details, we consider a concrete case. Analogously to Section \ref{sec:model}, we assume that $f_1=B-C, f_2 = -C, f_3 = B, f_4=0$ for some $0<C<B<1$. 
\begin{remark}
    We can make such assumptions on the parameters and they do not come into conflict with the conditions (\ref{eq:conditions}). Indeed, we can solve for $a,b,c,d$ from REF
    \begin{equation}
\begin{cases}
b=-a+B,\\
c=  a+C,\\
d= -a-C-B
\end{cases}
    \end{equation}
    It can be easily verified that, thus defined, $c>a$ and $c-a=  C<b-d=B$.
\end{remark}
With these assumptions, the adaptive dynamics take the form 
\begin{small}
    \begin{equation}
        \begin{cases}
        \label{eq: main sustem equatinos}
\dot{p}_1=\frac{1}{A}\Bigg(p_3 p_4 \Big(B (p_2-p_4) (p_1 (-p_4)+p_1+p_2 (p_3-1)-p_3+p_4)+b (p_2-p_1)\\+C\left(p_4 (-2 p_1 p_2+2 p_2 p_3+p_2+1)+(p_2-1) (p_1 p_2-p_2 p_3-1)+p_4^2 (p_1-p_3-1)\right)\Big)\Bigg),\\
\dot{p}_2 =  \frac{1}{A}\Bigg((1-p_1) p_4 \Big(B (p_3 (p_1 p_2-p_2 p_3-1)+p_1 p_4 (p_3-p_1)+p_4)+C \Big(p_1^2 (p_2-p_4-1)\\+
p_1 p_3 (-2 p_2+2 p_4+1)+p_3 (p_3+1) (p_2-p_4)-p_2+p_4+1\Big)\Big)\\
\dot{p}_3 = -\frac{(p_1-1) p_4}{A} \Bigg(b (p_2-p_4) (p_1 (-p_4)+p_1+p_2 (p_3-1)-p_3+p_4)+B(p_2-p_1)\\+C \Big(p_4 (-2 p_1 p_2+2 p_2 p_3+p_2+1)+(p_2-1) (p_1 p_2-p_2 p_3-1)+p_4^2 (p_1-p_3-1)\Big)\Bigg),\\
\dot{p}_4 = \frac{(1-p_1) (p_2-1)}{A} \Bigg(B \Big(\left(p_1^2-1\right) p_4-p_1 p_3 (p_2+p_4)+p_2 p_3^2+p_3\Big)\\+C \Big(p_1^2 (-p_2+p_4+1)+p_1 p_3 (2 p_2-2 p_4-1)-p_3 (p_3+1) (p_2-p_4)+p_2-p_4-1\Big)\Bigg)
        \end{cases}
    \end{equation}
\end{small}
where 
\begin{equation}
    \begin{split}
        A&=(p_1 (p_2-1)-2 p_2 p_3+p_2+p_3 p_4-1) (p_1 (p_2-2 p_4-1)-p_2+(p_3+2) p_4+1)^2.
    \end{split}
\end{equation}
The natural place to start our investigations is to determine the location of the equilibria. 
This is accomplished by equating the left hand side of the equations to zero. Consequently, we get
\begin{small}
    \begin{equation}
    \label{eq:equilibria}
    \begin{aligned}[c]
            &\begin{cases}
               p_1= 1,\\
               p_3= \frac{(p_2-1) (B-C) (p_2-p_4-1)}{B (p_2-1) (p_2-p_4)-C \left(-2 p_2 p_4+(p_2-1) p_2+p_4^2\right)}
               \end{cases}\\
               &\begin{cases}
              p_1=\frac{(B-C) p_2+C(p_4+1)}{B},\\
              p_3= \frac{C(1-p_2) + p_4(B+C)}{B}
               \end{cases}\\
               &\begin{cases}
              p_2= \frac{B p_3+C\left(p_1^2-p_1 p_3-1\right)}{B p_3 (p_1-p_3)+C \left((p_1-p_3)^2+p_3-1\right)}\\
              p_4= 0,
               \end{cases}\\
        \end{aligned}
        \begin{aligned}[c]
          &  \begin{cases}
                p_1 = \frac{C p_4}{B}+1\\
                p_2=1\\
                p_3=\frac{p_4 (B+C)}{B}
            \end{cases} \\
        &    \begin{cases}
                p_2= \frac{C (p_1-1)}{B}+p_1\\
                p_3= 0\\
                p_4 = \frac{C (p_1-1)}{B }
            \end{cases}
        \end{aligned}
    \end{equation}
\end{small}

Observe that the solutions in the left column are in the $[0,1]^4$ cube, whereas the ones in the right column are outside it or on the boundary.  
Focusing or attention on the right column, we can further observe that the first solution lies strictly outside $[0,1]^4$ while the second one  degenerates into $p_1 = 1,  \ p_2 = =1,  \ p_3=0, \ p_4 = 0$, since $p_1-1\le 0$ for all $p_1$ in the cube. 

Thus, the only solution that intersects $(0,1)^4$ is the second one in the left column; these equilibria form a two-dimensional plane parametrised by $p_2$ and $p_4$.

From this point, we can assume without loss of generality that $B=1$. This simplifying allows us to determine the linear stability of the solutions of the form
\begin{equation}
\label{eq: equilibrium}
\begin{cases}
    p_1 = (1-C)p_2 + + C(p_4-1)\\
    p_3 = C(1-p_2) + p_4(1 + C).
\end{cases}
\end{equation}

To do so, we consider the Jacobian of the system \eqref{eq: main sustem equatinos} and calculate the eigenvalues at the points belonging to the set \eqref{eq: equilibrium}.

The two eigenvalues $\lambda_1(p_2,p_4,c)$ and $
\lambda_2(p_2,p_4,c)$ are then given by 
\begin{small}
\begin{equation}
    \label{eq: eigenvalues}
    \begin{cases}
    \lambda_1 &=   -\frac{1}{F}\Bigg(2 c^2 (p_2-p_4-1) \left(p_2^2 (p_4-1)-2 p_2 \left(p_4^2+p_4-1\right)+p_4^3+p_4-1\right) \\&+\Big((c (p_2-1)^2 p_4 \left(c^2-2 c p_2-3\right)+(c-1) (p_2-1)^3 \left(-c^2+c+p_2+1\right)\\&-p_4^3 (c (c (c+6 p_2-2)+4 p_2+1)+2)+c (p_2-1) p_4^2 (c (c+6 p_2-2)+3)+(c+1) (2 c+1) p_4^4)^2\\&-8 c \left(c^2-1\right) p_4 (p_2-p_4+1) (-p_2+p_4+1)^2 ((c-1) (p_2-1)-c p_4) \left(c \left((p_2-1)^2-p_4^2\right)-2 (p_2-1) p_4\right)\Big)^{\frac12}\\&+c^3 (-p_2+p_4+1)^2 (p_2+p_4-1)-c \left(-(4 p_2+1) p_4^3+3 (p_2-1) p_4^2-3 (p_2-1)^2 p_4+(p_2-1)^3 p_2+3 p_4^4\right)\\&+p_2^4-2 p_2^3+2 p_2-p_4^4+2 p_4^3-1\Bigg),\\
    \lambda_2 &= \frac{1}{F}\Bigg(-2 c^2 (p_2-p_4-1) \left(p_2^2 (p_4-1)-2 p_2 \left(p_4^2+p_4-1\right)+p_4^3+p_4-1\right)\\&+\Big((c (p_2-1)^2 p_4 \left(c^2-2 c p_2-3\right)+(c-1) (p_2-1)^3 \left(-c^2+c+p_2+1\right)\\&-p_4^3 (c (c (c+6 p_2-2)+4 p_2+1)+2)+c (p_2-1) p_4^2 (c (c+6 p_2-2)+3)+(c+1) (2 c+1) p_4^4)^2\\&-8 c \left(c^2-1\right) p_4 (p_2-p_4+1) (-p_2+p_4+1)^2 ((c-1) (p_2-1)-c p_4) \left(c \left((p_2-1)^2-p_4^2\right)-2 (p_2-1) p_4\right)\Big)^{\frac12}\\&-c^3 (-p_2+p_4+1)^2 (p_2+p_4-1)+c \left(-(4 p_2+1) p_4^3+3 (p_2-1) p_4^2-3 (p_2-1)^2 p_4+(p_2-1)^3 p_2+3 p_4^4\right)\\&-p_2^4+2 p_2^3-2 p_2+p_4^4-2 p_4^3+1\Bigg)
    \end{cases}
\end{equation}
\end{small}
where 
\begin{equation}
    \begin{split}
        F:=2 (c-1)^2 (c+1) (p_2-p_4-1)^5 (p_2-p_4+1).
    \end{split}
\end{equation}

Due to the complexity of equations, we proceed numerically. Using Wolfram Mathematica, it can be established that the expression under the square root is positive for all values of $p_2,p_4,c\in(0,1)^3$. Therefore, the eigenvalues in \eqref{eq: eigenvalues} can be saddles, sources or sinks. 

Additional numerics demonstrate that $\lambda_1>0$ for all admissible $p_2,p_4,c$ -- this excludes sinks. 

The value $\lambda_2$, however, can be positive or negative, depending on the variables. This demonstrates
\begin{lemma}
    The plane of equilibria \eqref{eq: equilibrium} consists of  degenerate saddles and degenerate sources. By degeneracy here we mean that two eigenvalues out of four are zero.
\end{lemma}

At this point, looking at the equations \eqref{eq: main sustem equatinos}, we can make an important observation that distinguishes the alternating case from the simultaneous one.  

\begin{theorem}
For any values of $f_1,f_2,f_3, f_4$ the dynamics have the two following invariant quantities:
\begin{enumerate}
    \item $F_1 := (p_1-1)^2 + p_3^2$,
    \item $F_
    2:= (p_2-1)^2 + p_4^2$.
\end{enumerate}
\end{theorem}
\begin{proof}
    This can be  demonstrated by taking the scalar products of the left hand side of (\ref{eq: main sustem equatinos}) and $\nabla F_1$,  $\nabla F_2$ and ascertaining that both of them are equal to 0. 
\end{proof}
\begin{remark}
It appears that these quantities are not connected to any symmetries of the system -- functions like these are typically related to rotational symmetries (specifically, rotating in the to -dimensional planes formed by $p_1$ and $p_3$, as well as $p_2$ and $p_4$). However, dynamics does not possess such symmetries. 
\end{remark}
\begin{corollary}
\label{corollary: toric coordinates}
     The dynamics of (\ref{eq: main sustem equatinos}) take place on a two-dimensional torus $\mathbf{T}^2$. 
\end{corollary}
\begin{proof}
    Fixing $C_1$, the value of $F_1$ and $C_2$, the value of $F_2$, we observe that $p_1$ and $p_3$ lie on a circle of radius $\sqrt{C_1}$, centred at $(1,0)$. The same can be said for the variables $p_2$ and $p_4$, albeit the radius of the circle will be $\sqrt{C_2}$. Since the functions $F_1$ and $F_2$ have no variables in common, their common level set in  $\mathbb{R}^4$ will be the direct product of their individual level sets. that is, $S^1\times S^1 = \mathbf{T}^2$. 
\end{proof}
\begin{corollary}
    Since the motion is restricted to tori, the dynamics are globally bounded. 
\end{corollary}

Using these observations, we  no longer need to consider the entire phase space. Observe that for varying $C_1$ and $C_2$ the four-dimensional space $\mathbb{R}^4 $ foliates into two-dimensional tori. Hence, to study motion, we need to fix $C_1$ and $C_2$ and  investigate the motion on thusly obtained torus.

\subsection{Motion on tori}
\begin{figure}
    \centering
 \subfigure[$c=0.31, C_1 = 0.355, C_2= 0.314$]{\includegraphics[scale =.473]{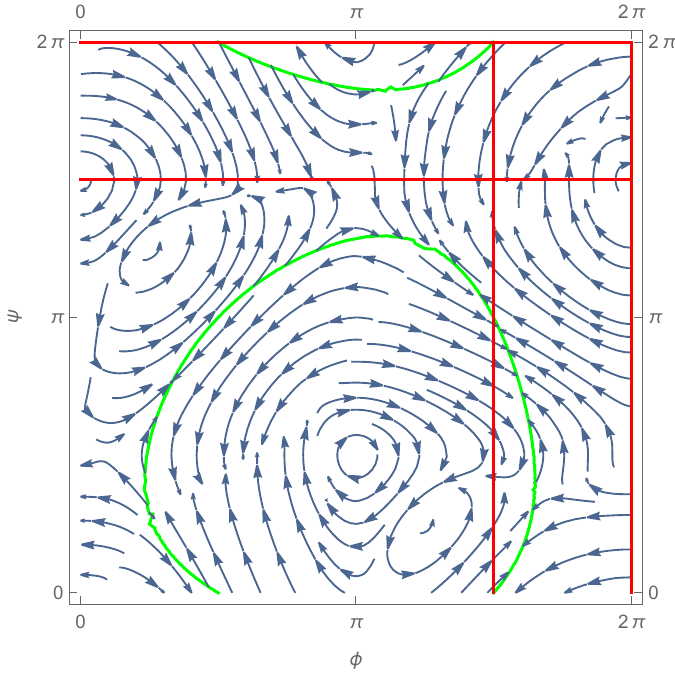}}
 \subfigure[$c=0.4, C_1 = 1.16422, C_2 = 1.158$]{\includegraphics[scale=.473]{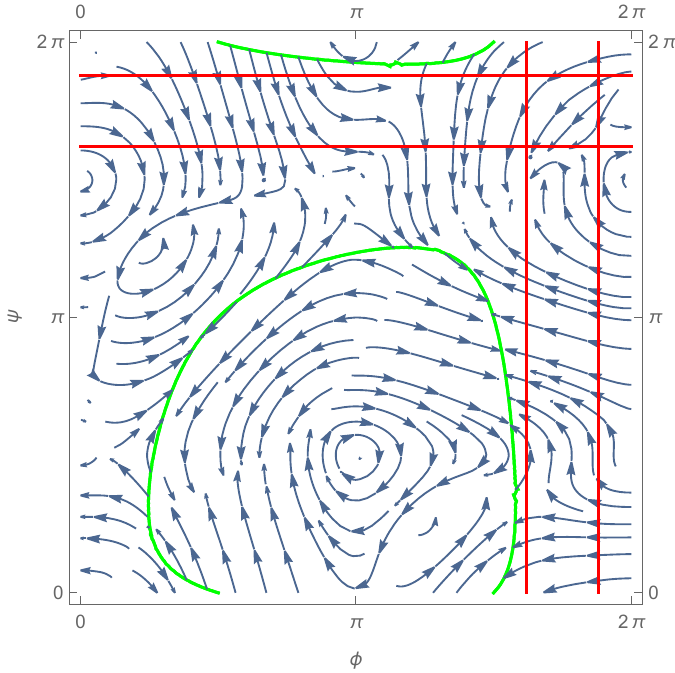}}
 \caption{Vector field on a torus for various value of $c, C_1, C_2$. Red curves are the zero of the denominator of (\ref{eq: equations on tori general}), and the quare formed by four lines is the image of the unit square. 
 }
    \label{fig:motion on tori}
\end{figure}

The main positive effect that existence of invariants has on a dynamical system is the possibility for reduction. This can be done in a number of ways: the most obvious one would be expressing, say, $p_3$ and $p_4$ from the system of equations
\[
\begin{cases}
    (p_1-1)^2 + p_3^2 = C_1\\
    (p_2-1)^2 + p_4^2=C_2.
\end{cases}
\]
This method requires using square roots -- however, the choice of signs is simple, as all $p_i$ must be positive. However, this method yields the equations that are too complex. 

Another way of reducing the motion is through expressing the variables $p_i$ via angles on the torus. This method is particularly illuminating as is does not distort the geometry of the system by ``flattening" it to the two-dimensional plane. Additionally, in order to restore the full motion, we only need to glue the opposite edges of the square. 

With these considerations in mind, we fix $C_1$ and $C_2$ and parametrise the torus by two angles $\phi$ and $\psi$, such that $\phi,\psi\in[0,2\pi]$ and 
\begin{equation}
\label{eq: change of variables angles}
    \begin{cases}
    p_1= 1+ \sqrt{C_1} \sin (\phi ),\\
   p_2=  1 + \sqrt{C_2} \sin (\psi ),\\
   p_3=  \sqrt{C_1}\cos (\phi)
   \\ p_4=\sqrt{C_2} \cos (\psi )
    \end{cases}
\end{equation}
After appropriate substitutions, we get a system of two equations reading 
\begin{small}
\begin{equation}
\label{eq: equations on tori general}
    \begin{cases}
      \dot{\phi} &=  \frac{\cos (\psi )}{C_1 G} \Bigg(2 \sqrt{C_1} \Big(-\sqrt{C_2} \sin (2 \psi ) (2 C \sin (\phi )-2 C \cos (\phi )+\sin (\phi )+\cos (\phi ))\\&+\sqrt{C_2} (2 C \sin (\phi )-2 C \cos (\phi )+\sin (\phi )+\cos (\phi ))\\&+2 \sin (\psi ) (c \sin (\phi )-C \cos (\phi )+\sin (\phi )+\cos (\phi ))+\sqrt{C_2} \cos (2 \psi ) (\sin (\phi )-\cos (\phi ))\Big)+(C-1) \sqrt{C_2}\Bigg)\\&+8 \sqrt{C_1} \cos ^2(\psi ) (c \cos (\phi )-(C+1) \sin (\phi ))-(C-1) \sqrt{C_2} (\sin (\psi )+\sin (3 \psi )+\cos (3 \psi )),\\
       \dot{\psi}&= \frac{\sin (\phi )}{C_2 G} \Bigg(-\sqrt{C_2} \cos (\psi ) \Big(\sqrt{C_1} ((2 C+1) \sin (2 \phi )+\cos (2 \phi ))-(2 C+1) \sqrt{C_1}\\&-4 (C+1) \sin (\phi )+2 (C+1) \cos (\phi )\Big)+\sqrt{C_2} \sin (\psi ) \Big((2 C-1) \sqrt{C_1} \sin (2 \phi )\\&+\sqrt{C_1} (-2 C+\cos (2 \phi )+1)-4 c \sin (\phi )+2 (C-1) \cos (\phi )\Big)-(C-1) \sqrt{C_1} (-\sin (2 \phi )+\cos (2 \phi )+1)\Bigg)
    \end{cases}
    \end{equation}
\end{small}
where
\begin{small}
\begin{equation}
\begin{split}
G&= 4 (\cos (\psi ) \cos (\phi )+\sin (\phi ) (\sin (\psi )-2 \cos (\psi )))^2\times\\& \Bigg(\sqrt{C_1} \cos (\phi ) \Big(\sqrt{C_2} (\cos (\psi )-2 \sin (\psi ))-2\Big)+\sqrt{C_2} \sin (\psi ) \Big(\sqrt{C_1} \sin (\phi )+2\Big)\Bigg)
\end{split}
\end{equation}
\end{small}
This is indeed a vector field on a torus, as shown in Figure \ref{fig:motion on tori}.  We will discuss the figure in detail later. 

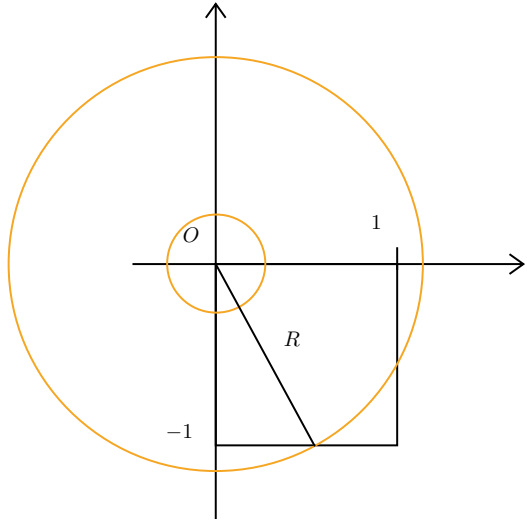
\begin{figure}
    \centering
\tikzset{every picture/.style={line width=0.75pt}} 

\begin{tikzpicture}[x=0.75pt,y=0.75pt,yscale=-1,xscale=1]

\draw  (165.83,144.17) -- (362.98,144.17)(207.83,12.85) -- (207.83,272.85) (355.98,139.17) -- (362.98,144.17) -- (355.98,149.17) (202.83,19.85) -- (207.83,12.85) -- (212.83,19.85)  ;
\draw   (207.83,144.17) -- (299.33,144.17) -- (299.33,235.67) -- (207.83,235.67) -- cycle ;
\draw  [color={rgb, 255:red, 245; green, 166; blue, 35 }  ,draw opacity=1 ] (183.33,143.92) .. controls (183.33,130.25) and (194.41,119.17) .. (208.08,119.17) .. controls (221.75,119.17) and (232.83,130.25) .. (232.83,143.92) .. controls (232.83,157.59) and (221.75,168.67) .. (208.08,168.67) .. controls (194.41,168.67) and (183.33,157.59) .. (183.33,143.92) -- cycle ;
\draw  [color={rgb, 255:red, 245; green, 166; blue, 35 }  ,draw opacity=1 ] (103.37,144.17) .. controls (103.37,86.48) and (150.14,39.71) .. (207.83,39.71) .. controls (265.52,39.71) and (312.29,86.48) .. (312.29,144.17) .. controls (312.29,201.86) and (265.52,248.63) .. (207.83,248.63) .. controls (150.14,248.63) and (103.37,201.86) .. (103.37,144.17) -- cycle ;
\draw    (207.83,144.17) -- (257.83,236.17) ;
\draw    (299.33,135.67) -- (299.33,147.17) ;

\draw (241,177) node [anchor=north west][inner sep=0.75pt]  [xscale=0.8,yscale=0.8] [align=left] {$\displaystyle R$};
\draw (285,118.5) node [anchor=north west][inner sep=0.75pt]  [xscale=0.8,yscale=0.8] [align=left] {$\displaystyle 1$};
\draw (190,124.5) node [anchor=north west][inner sep=0.75pt]  [xscale=0.8,yscale=0.8] [align=left] {$\displaystyle O$};
\draw (181.5,224) node [anchor=north west][inner sep=0.75pt]  [xscale=0.8,yscale=0.8] [align=left] {$\displaystyle -1$};

\end{tikzpicture}
\caption{Possible values of $\phi$ and $\psi$}
\label{fig: values of angles}
\end{figure}

We observe two things. Firstly, due to the limitation $\p\in[0,4]^4$ we are no interested in all tori, only in those that intersect the four dimensional cube. 

Secondly, even for the tori that do intersect the cube, we need not consider the motion on their entire surface. The exact form of the region we will be considering is the ropc of discussion in the section below. 

\subsubsection{Intersection of tori with $[0,1]^4$ and existence of solutions}

From the substitution (\ref{eq: change of variables angles}) we note that  the following must hold:
\[
\begin{cases}
    \sqrt{C_1}\sin(\phi)\in(-1,0)\\
     \sqrt{C_1}\cos(\phi)\in(0,1),\\
     \sqrt{C_2}\sin(\psi)\in(-1,0),\\ \sqrt{C_2}\cos(\psi)\in(0,1).
\end{cases}
\]

We perform explicit calculations for $\phi$ and $C_1$, and then identical deductions can be made for $\psi$ and $C_2$. 

The two expressions $\sqrt{C_1}\sin(\phi)$ and $\sqrt{C_1}\cos(\phi)$ belong to a circle (for a geometric illustration, see Figure \ref{fig: values of angles}). We deduce that the variables must be in the fourth quadrant of the plane. 

In order for the intersections to exist, note that $0<C_1<2$ must hold. With varying $C_1$, the intervals  of  admissible values of $\phi$ can be of two types:
\begin{enumerate}
    \item when $0<C_1<1$ (inner orange cirle ) the admissible interval is $\left[\frac{3\pi}{2}, 2\pi\right]$;
    \item when $1<C_1<2$ (larger orange circle), the admissible interval corresponds to the values of the angle between the two intersections of the square with the circle. 
\end{enumerate} 
After performing elementary calculations (which we omit here for brevity), we obtain

\begin{lemma}
\label{st: lemma intervals of angles}
After the coordinate change as in (\ref{eq: change of variables angles}), the image of the 4-cube is one of the following rectangles:
\begin{enumerate}
    \item $\phi\in\left(\frac{3\pi}{2}, 2\pi\right),\psi\in\left(\frac{3\pi}{2}, 2\pi\right)$, if $C_1,C_2\in(0,1]$;
    \item  $\phi\in\left(\frac{3\pi}{2}, 2\pi\right),\psi\in\left(2 \pi -\cos ^{-1}\left(\frac{1}{\sqrt{C_2}}\right), 2 \pi -\sin ^{-1}\left(\frac{1}{\sqrt{C_2}}\right)\right)$, if $C_1\in(0,1],C_2\in(1,2]$;
    \item  $\phi\in\left(2 \pi -\cos ^{-1}\left(\frac{1}{\sqrt{C_1}}\right), 2 \pi -\sin ^{-1}\left(\frac{1}{\sqrt{C_1}}\right)\right),\psi\in\left(\frac{3\pi}{2}, 2\pi\right)$, if $C_1\in(1,2],C_2\in(0,1)$;
    \item  $\phi\in\left(2 \pi -\cos ^{-1}\left(\frac{1}{\sqrt{C_1}}\right), 2 \pi -\sin ^{-1}\left(\frac{1}{\sqrt{C_1}}\right)\right),\ $ $\psi\in\left(2 \pi -\cos ^{-1}\left(\frac{1}{\sqrt{C_2}}\right), 2 \pi -\sin ^{-1}\left(\frac{1}{\sqrt{C_2}}\right)\right)$, if $C_1,C_2\in(1,2]$;
\end{enumerate}
\end{lemma}

Note that these rectangles do not intersect the curves that describe zeros of the denominator-- this follows from denominator having no zeros inside $(0,1)^4$. 

The images of the four dimensional cube  are shown in Figure \ref{fig:motion on tori} as the insides of the rectangles with red boundaries. As it can be observed from Figure \ref{fig:motion on tori}, the little rectangle does not always contain an equilibrium However, by varying $C_1$ and $C_2$ we can make equilibria appear. 

This begs the following question: what restrictions do we need to place on $C_1,C_2$ in order for equilibria to belong to the rectangle?

The equations (\ref{eq: equations on tori general}) are hard to solve explicitly; we circumvent that by substituting the expressions (\ref{eq: change of variables angles}) into the pre-calculated equations for equilibria in (\ref{eq:equilibria}), in particular, the second one on the right. This yields
\begin{equation}
    \label{eq: equilibiria tori}
    \begin{cases}
       \sqrt{C_1} \sin (\phi )=\sqrt{C_2} (-C \sin (\psi )+C \cos (\psi )+\sin (\psi ))\\
        \sqrt{C_1} \cos (\phi )=\sqrt{C_2} ((C+1) \cos (\psi )-C \sin (\psi ))
    \end{cases}
\end{equation}
These equations can be transformed in the following way: first, subtract the second equation form the first  them up to obtain
\begin{equation}
    \label{eq: tori equation 1}
    \sqrt{C_1}\left(\sin(\phi) - \cos(\phi)\right) = \sqrt{C_2}\left(\sin(\psi) -\cos(\psi)\right).
\end{equation}
Dividing by $\sqrt{2}$ and using the identity 
\[\frac{1}{\sqrt{2}}\sin(x) - \frac{1}{\sqrt{2}}\cos(x) = \sin\left(x - \frac{\pi}{4}\right),\]  we can transform (\ref{eq: tori equation 1}) into
\begin{equation}
    \label{eq: equation 1 tori final form}
    \sin\left(\phi-\frac{\pi}{4}\right)  = \sqrt{\frac{C_2}{C_1}}\sin\left(\psi-\frac{\pi}{4}\right).
\end{equation}
Another way of simplifying the equations is squaring both and summing up, to obtain
\[
C_1 = C_2\left(\left(C \cos(\psi) + (1-C)\sin(\psi)\right)^2 + \left((C+1)\cos(\psi) - C\sin(\psi)\right)^2\right)
\]
or
\[
2 C^2 C_2+2 C C_2 (\cos (2 \psi )-C \sin (2 \psi ))-C_1+C_2=0. 
\]
To solve this, we use a similar trick and divide both parts of the equation by $\sqrt{C^2+1}$, obtaining
\begin{equation}
    \begin{split}
         \frac{C}{\sqrt{1+C^2}}\sin(2\psi) -\frac{1}{\sqrt{1+C^2}}\cos(2\psi)  &= \frac{C_2+2C_2C^2-C_1}
        {2C\sqrt{1+C^2}C_2};\\
        \sin\left(2\psi  - \arccos\left(\frac{C}{\sqrt{1 +C^2}}\right)\right) &=\frac{C_2+2C_2C^2-C_1}
        {2C\sqrt{1+C^2}C_2}\\
        \end{split}
        \end{equation}
Note that since both $\frac{C}{\sqrt{1 + C^2}}$ and $\frac{1}{\sqrt{1 + C^2}}$ are positive expressions, we can assume that $\arccos\left(\frac{C}{\sqrt{1 + C^2}}\right)\in\left(0,\pi/2\right)$ and is uniquely defined.  Solving this explicitly yields
        \begin{equation}
        \label{eq: psi solution}
            \begin{split}
        2\psi &=\begin{sqcases}
\arcsin\left(\frac{C_2+2C_2C^2-C_1}
        {2C\sqrt{1+C^2}C_2}\right) &+ \arccos\left(\frac{C}{\sqrt{1+C^2}}\right) + 2\pi k, \ \ k\in\mathbb{Z}\\
-\arcsin\left(\frac{C_2+2C_2C^2-C_1}
        {2C\sqrt{1+C^2}C_2}\right) &+ \arccos\left(\frac{C}{\sqrt{1+C^2}}\right) + \pi(2 k+1), \ \ k\in\mathbb{Z}        \end{sqcases} \\
        &\mathrm{or}\\
        \psi &=\begin{sqcases}
        \frac{1}{2}\Bigg(\arcsin\left(\frac{C_2+2C_2C^2-C_1}
        {2C\sqrt{1+C^2}C_2}\right) + \arccos\left(\frac{C}{\sqrt{1+C^2}}\right)\Bigg) + \pi k , \ \ k\in\mathbb{Z}\\
     \frac{1}{2}\Bigg(-\arcsin\left(\frac{C_2+2C_2C^2-C_1}
        {2C\sqrt{1+C^2}C_2}\right) + \arccos\left(\frac{C}{\sqrt{1+C^2}}\right)\Bigg) + \frac{(2 k+1)}{2}\pi, \ \ k\in\mathbb{Z}  
        \end{sqcases}
    \end{split}
\end{equation}

Both the last expression from (\ref{eq: psi solution}) and the consequent value of $\phi$
\begin{equation}
\label{eq: value of phi}
\phi = \begin{sqcases}
\frac{\pi}{4}  + \arcsin\left(\sqrt{\frac{C_2}{C_1}}\sin\left(\psi-\frac{\pi}{4}\right)\right) + 2\pi l,\ \ l\in\mathbf{Z}\\
\frac{\pi}{4}  - \arcsin\left(\sqrt{\frac{C_2}{C_1}}\sin\left(\psi-\frac{\pi}{4}\right)\right) + \pi (2l+1),\ \ l\in\mathbf{Z}
\end{sqcases}
\end{equation}
must belong to one of the permitted intervals from Lemma \ref{st: lemma intervals of angles}. 

Leveraging in this, we can summarise:
\begin{lemma}
    \label{st: lemma wwhen solutions lie in square}
    For given values of $C_1, C_2$ and $C$ equilibria exist in the image of $(0,1)^4$ if exist $k,l\in\mathbf{Z}$ such that  the value in (\ref{eq: psi solution}),\ref{eq: value of phi}) belong to correponding intervals from Lemma \ref{st: lemma intervals of angles}. 
\end{lemma}

These observations, as well as elementary properties of the $\arcsin$ function, lead to: 
\begin{corollary}
    For any given value of $C_1, C_2$ and $c$ there can be no more than four equilibria inside the square from Lemma \ref{st: lemma intervals of angles}. 
\end{corollary}
\begin{proof}
    This is clear from the form of the equations, as well as the fact that one solution in $\psi$ can correspond to two solutions in $\phi$. 
\end{proof}

    
\begin{remark}
    We have already discussed the stability of the equilibria above, in Section \ref{sec: memory 1}. As one can see, this is much harder to do in toric coordinates. 
\end{remark}
\subsubsection{Degenerate Tori}
Lastly, we want to quickly address the motion on degenerate tori, i.e. the ones that have one of the parameters $C_1$ or $C_2$ very small compared to the other. 

One can re-parametrise the motion on the 
It can be checked that for the original system \ref{eq: equations on tori general} the degenerate tori consist entirely of singularities. Therefore, prior to considering the behaviour of the system, one needs to desingularise it via multiplication by the denominator. Performing the same coordinate changes, we arrive at the same result that we would have achieved by multiplying the system \eqref{eq: equations on tori general} by the smallest common denominator of the denominators of the two equations. In this section, we will by default consider the \textbf{desingularised} system.\\

It has the form
\begin{equation}
    \label{eq: desingularised toric}
    \begin{split}
    \begin{cases}
   \dot{\phi} =&  \frac{1}{2} C_2 \cos (\psi ) \Bigg(2 \sqrt{C_1} \cos (\phi ) \Big(\sin (\psi ) \left(-C+\sqrt{C_2} \sin (\psi )+1\right)\\&+\cos (\psi ) \left(2 C-(1-2 C) \sqrt{C_2} \sin (\psi )\right)-C \sqrt{C_2}\Big)\\&+\sqrt{C_1} \sin (\phi ) \Big(-(2 C+1) \sqrt{C_2} \sin (2 \psi )+(2 C+1) \sqrt{C_2}+2 (C+1) \sin (\psi )\\&-4 (C+1) \cos (\psi )+\sqrt{C_2} \cos (2 \psi )\Big)+(1-C) \sqrt{C_2} (\sin (2 \psi )+\cos (2 \psi )-1)\Bigg),\\
   \dot{\psi}=&-\sqrt{C_1} \sin (\phi ) \Bigg(-C_1 \cos ^2(\phi ) \left(\sqrt{C_2} ((1-C) \sin (\psi )+c \cos (\psi ))-C+1\right)\\&+C_1 \sin (\phi ) \cos (\phi ) \Big(\sqrt{C_2} ((1-2 C) \sin (\psi )+(2 C+1) \cos (\psi ))-C+1\Big)\\&+\sqrt{C_1} \sqrt{C_2} \cos (\phi ) ((1-C) \sin (\psi )+(C+1) \cos (\psi ))\\&-\sqrt{C_1} \sqrt{C_2} \sin (\phi ) \left(\sqrt{C_1} \sin (\phi )+2\right) \big((C+1) \cos (\psi )-C \sin (\psi )\big)\Bigg)
   \end{cases}
    \end{split}
\end{equation}
\noindent
This is a system on $\mathbb{T}^2$ that has a number of degenerate equilibrium points. \\
\noindent
$\mathbf{C_1}=0$. \\
We primarily address the case of $C_1=0$; it can be seen that when $C_2=0$, the situation is completely analogous.\\

The second equation of \eqref{eq: desingularised toric} has the form $C_1 F_1(\phi, \psi, C_2) + C_1^{\frac32}F_2(\phi,\psi, C_2)$. Therefore, both equations have the form of the so-called \textit{fast} dynamics:
\[
\begin{cases}
\dot{\phi} = G_1(\phi,\psi,C_2) + \sqrt{C_1}G_2(\phi,\psi,C_2)\\
\dot{\psi} = C_1\left(F_1(\phi,\psi,C_2) + \sqrt{C_1}F_2(\phi,\psi,C_2)\right).
\end{cases}
\]
Therefore, setting $C_1\approx 0$ should give us an insight into the dynamics of the system.\\
\begin{figure}
    \centering

\tikzset{every picture/.style={line width=0.75pt}} 

\begin{tikzpicture}[x=0.75pt,y=0.75pt,yscale=-1,xscale=1]

\draw    (92.83,119.67) .. controls (122.83,70.83) and (272.83,77.33) .. (288.83,118.83) ;
\draw   (75,117.08) .. controls (75,91.08) and (127.23,70) .. (191.67,70) .. controls (256.1,70) and (308.33,91.08) .. (308.33,117.08) .. controls (308.33,143.09) and (256.1,164.17) .. (191.67,164.17) .. controls (127.23,164.17) and (75,143.09) .. (75,117.08) -- cycle ;
\draw    (90.83,117.17) .. controls (110.83,148.17) and (257.33,164.67) .. (291.83,114.17) ;
\draw   (105.18,138.08) .. controls (108.3,134.08) and (112.47,131.73) .. (114.49,132.84) .. controls (116.51,133.94) and (115.61,138.09) .. (112.49,142.09) .. controls (109.37,146.09) and (105.2,148.44) .. (103.18,147.33) .. controls (101.16,146.22) and (102.05,142.08) .. (105.18,138.08) -- cycle ;
\draw   (219.12,149.16) .. controls (218.96,145.11) and (220.16,144.15) .. (221.8,147.01) .. controls (223.43,149.86) and (224.89,155.46) .. (225.05,159.51) .. controls (225.2,163.55) and (224.01,164.52) .. (222.37,161.66) .. controls (220.74,158.8) and (219.28,153.21) .. (219.12,149.16) -- cycle ;
\draw   (110.83,136.83) .. controls (111.02,140.44) and (110.49,143.55) .. (109.26,146.15) .. controls (111.2,144.05) and (113.86,142.46) .. (117.22,141.38) ;
\draw   (229.83,154.33) .. controls (226.72,152.94) and (224.5,151.23) .. (223.2,149.19) .. controls (223.6,151.56) and (223.09,154.25) .. (221.68,157.27) ;
\draw   (284.67,130.21) .. controls (281.45,126.78) and (281.61,124) .. (285.05,124) .. controls (288.48,124) and (293.87,126.78) .. (297.09,130.21) .. controls (300.32,133.65) and (300.15,136.43) .. (296.72,136.43) .. controls (293.29,136.43) and (287.89,133.65) .. (284.67,130.21) -- cycle ;
\draw   (297.75,124.93) .. controls (294.47,125.89) and (291.68,126.02) .. (289.36,125.3) .. controls (291.2,126.85) and (292.55,129.23) .. (293.43,132.45) ;
\draw   (116.52,88.08) .. controls (113.53,84.63) and (113.91,81.83) .. (117.36,81.83) .. controls (120.81,81.83) and (126.03,84.63) .. (129.02,88.08) .. controls (132,91.54) and (131.62,94.33) .. (128.17,94.33) .. controls (124.72,94.33) and (119.5,91.54) .. (116.52,88.08) -- cycle ;
\draw   (239.01,83.5) .. controls (241.42,79.73) and (246.43,76.67) .. (250.21,76.67) .. controls (253.98,76.67) and (255.09,79.73) .. (252.68,83.5) .. controls (250.27,87.27) and (245.26,90.33) .. (241.49,90.33) .. controls (237.71,90.33) and (236.61,87.27) .. (239.01,83.5) -- cycle ;
\draw   (249.69,92.66) .. controls (249.87,89.25) and (250.67,86.56) .. (252.1,84.61) .. controls (250.04,85.84) and (247.35,86.34) .. (244.01,86.11) ;
\draw   (120.03,86.55) .. controls (123.39,85.91) and (126.18,86.05) .. (128.42,86.98) .. controls (126.74,85.27) and (125.61,82.77) .. (125.05,79.48) ;

\end{tikzpicture}
    \caption{Motion on almost degenerate tori. Without loss of generality, we may assume that the ``smaller" circles have radius $C_1$ and the ``larger" circles have radius $C_2$. The motion will be on the trajectories that are very close to the small circles.}
    \label{fig:placeholder}
\end{figure}
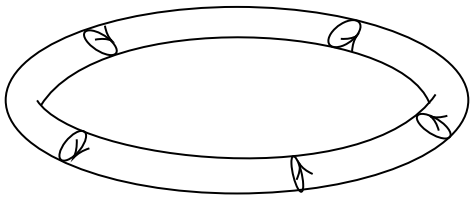
The result of making $C_1\to 0$ is the so-called slow-fast dynamics: the part proportional to $C_1$ will be the slow part ($\dot{\psi}$), whereas the ``fast" part is the dynamics of $\phi$. However, $C_1$ is the radius of the circle in $\phi$. Therefore, the dynamics are going to move in the diminishing direction until it ``collapses". 
\begin{lemma}
    The trajectories on almost degenerate tori are aperiodic.
\end{lemma}
\begin{proof}
    The function-coefficient in $\dot{\psi}$ at $C_1$ has the form 
    \begin{equation*}
        \begin{split}
            F(\phi,\psi) =& -\sin (\phi ) \Big(\sqrt{C_2} \cos (\phi ) ((1-C) \sin (\psi )+(C+1) \cos (\psi ))\\&-2 \sqrt{C_2} \sin (\phi ) ((C+1) \cos (\psi )-C \sin (\psi ))\Big)
        \end{split}
    \end{equation*}
    Averaging this function over $\phi\in[0,2\pi]$ yields
    \[
\int\limits_0^{2\pi}F(\phi,\psi)\mathrm{d}\phi = 2 \pi  \sqrt{C_2} ((C+1) \cos (\psi )-c \sin (\psi ))
    \]
    This is not equal to $0$ for the majority of $\psi$; therefore, the motion is aperiodic. 
\end{proof}
\noindent
$\mathbf{C_2}=0$. \\
\noindent
This case is treated absolutely analogously, down to the powers of $C_2$ involved.  

\section{Conclusion}
In this paper, we considered an alternative decision making in  the prisoner's dilemma: the one where players take turns making decisions. 

While the setup of the game (the transition matrix, the payoff vector, etc.) is significantly different from the simultaneous case, the algorithms and the methods used to set up adaptive dynamics and to demonstrate some fundamental symmetries of the system are quite similar.

Thus, we were able to provide recursive constructions for the transition matrix and the payoff vector, describe the class of admissible orthogonal matrices and demonstrate that the adaptive dynamics possess the same discrete $\mathbf{Z}_2$-symmetry as the simultaneous ones do. 

However, the adaptive dynamics themselves are very different from the ones in  \cite{balabanova2024adaptive}. The main difference lies in existence of invariant quantities that have a simple explicit form and allow us to foliate the space into tori and thus reduce the motion.  We write explicit conditions for existence of equilibria on the tori, as well as give estimates on the number of euiqlibria inside the interaction of $(0,1)^4$ cube with  a given torus. Lastly, we consider the dynamics on degenerate tori and demonstrate that the trajectories of the system lie very close to perioic circular trajectories.

\end{document}